\documentclass[11pt]{article}

\usepackage{amssymb, amsmath, amsthm, graphicx}
\usepackage[left=1in,top=1in,right=1in]{geometry}

\date{}

\theoremstyle{plain}
      \newtheorem{theorem}{Theorem}[section]
      \newtheorem{lemma}[theorem]{Lemma}
            \newtheorem{claim}[theorem]{Claim}

      \newtheorem{corollary}[theorem]{Corollary}
      
      \newtheorem{conjecture}[theorem]{Conjecture}
\theoremstyle{definition}

\theoremstyle{remark}

\def\twr{\mbox{\rm twr}}

\title{Off-diagonal hypergraph Ramsey numbers}

\author{Dhruv Mubayi\thanks{Department of Mathematics, Statistics, and Computer Science, University of Illinois, Chicago, IL, 60607 USA.  Research partially supported by NSF grant DMS-1300138. Email: {\tt mubayi@uic.edu}} \and Andrew Suk\thanks{Department of Mathematics, Statistics, and Computer Science, University of Illinois, Chicago, IL, 60607 USA. Supported by NSF grant DMS-1500153. Email: {\tt suk@uic.edu}.}}

\begin{document}

\maketitle

\begin{abstract}

The Ramsey number $r_k(s,n)$ is the minimum $N$ such that every red-blue coloring of the $k$-subsets of $\{1, \ldots, N\}$ contains a red set of size $s$ or a  blue set of size $n$, where a set is red (blue) if all of its $k$-subsets are red (blue).
 A $k$-uniform \emph{tight path} of size $s$, denoted by $P_{s}$, is a set of $s$ vertices $v_1 < \cdots < v_{s}$ in $\mathbb{Z}$,  and all $s-k+1$ edges of the form $\{v_j,v_{j+1},\ldots, v_{j + k -1}\}$.
 Let $r_k(P_s, n)$ be the minimum $N$ such that every red-blue coloring of the $k$-subsets of $\{1, \ldots, N\}$ results in a red $P_{s}$  or a blue  set of size $n$. The problem of estimating both $r_k(s,n)$ and $r_k(P_s, n)$ for $k=2$ goes back to the seminal work of Erd\H os and Szekeres from 1935, while the case $k\ge 3$ was first investigated by Erd\H os and Rado in 1952.

In this paper, we deduce a quantitative relationship between multicolor variants of $r_k(P_s, n)$ and $r_k(n, n)$. This yields several consequences including the following:

\begin{itemize}

\item We determine the correct tower growth rate for both $r_k(s,n)$ and $r_k(P_s, n)$ for $s \ge k+3$. The question of determining the tower growth rate of $r_k(s,n)$ for all $s \ge k+1$ was posed by Erd\H os and Hajnal in 1972.

 \item We show that  determining the tower growth rate of $r_k(P_{k+1}, n)$ is
 equivalent to determining the tower growth rate of $r_k(n,n)$, which is a notorious conjecture of Erd\H os, Hajnal and Rado from 1965 that remains open.
 \end{itemize}
 Some related off-diagonal hypergraph Ramsey problems are also explored.

\end{abstract}

\section{Introduction}
 A $k$-uniform hypergraph $H$ with vertex set $V$ is a collection of $k$-element subsets of $V$.  We write $K^{(k)}_n$ for the complete $k$-uniform hypergraph on an $n$-element vertex set.  The Ramsey number $r_k(s,n)$ is the minimum $N$ such that every red-blue coloring of the edges of  $K^{(k)}_N$  contains a monochromatic red copy of $K_s^{(k)}$ or a monochromatic blue copy of $K^{(k)}_n$. Due to its wide range of applications in logic, number theory, analysis, geometry, and computer science, estimating Ramsey numbers has become one of the most central problems in combinatorics.

\subsection{Diagonal Ramsey numbers.}
 The \emph{diagonal} Ramsey number, $r_k(n,n)$ where $k$ is fixed and $n$ tends to infinity, has been studied extensively over the past 80 years.  Classic results of Erd\H os and Szekeres \cite{ES35} and Erd\H os \cite{E47} imply that $2^{n/2} < r_2(n,n) \leq 2^{2n}$ for every integer $n > 2$.  These bounds have been improved by various authors (see \cite{C09,T88,GR}).  However, the constant factors in the exponents have not changed over the last 70 years.  For 3-uniform hypergraphs, a result of Erd\H os, Hajnal, and Rado \cite{EHR} gives the best known lower and upper bounds for $r_3(n,n)$ which are of the form\footnote{We write $f(n) = O(g(n))$ if $|f(n)| \leq c|g(n)|$ for some fixed constant $c$ and for all $n \geq 1$; $f(n) = \Omega(g(n))$ if $g(n) = O(f(n))$; and $f(n) = \Theta(g(n))$ if both $f(n) = O(g(n))$ and $f(n) = \Omega(g(n))$ hold.  We write $f(n) = o(g(n))$ if for every positive $\epsilon > 0$ there exists a constant $n_0$ such that $f(n) \leq \epsilon |g(n)|$ for all $n \geq n_0$.} $2^{\Omega(n^2)} < r_3(n,n) \leq 2^{2^{O(n)}}.$  For $k \geq 4$, there is also a difference of one exponential between the known lower and upper bounds for $r_k(n,n)$, that is, $$\twr_{k-1}(\Omega(n^2)) \leq r_k(n,n) \leq \twr_k(O(n)),$$

\noindent where the tower function $\twr_k(x)$ is defined by $\twr_1(x) = x$ and $\twr_{i + 1} = 2^{\twr_i(x)}$ (see \cite{ES35,ER,EH72}).  A notoriously difficult conjecture of Erd\H os, Hajnal, and Rado states the following (Erd\H os offered \$500 for a proof).

\begin{conjecture} {\bf (Erd\H os-Hajnal-Rado~\cite{EHR})}
\label{conj1}
For $k\geq 3$ fixed, $r_k(n,n) \geq \twr_{k}(\Omega(n))$.
\end{conjecture}
The study of $r_3(n,n)$ may be particularly important for our understanding of hypergraph Ramsey numbers.  Any improvement on the lower bound for $r_3(n,n)$ can be used with a result of Erd\H os and
Hajnal \cite{EH72}, known as the stepping-up lemma, to obtain improved lower bounds for $r_k(n,n)$, for all $k\geq 4$.  In particular, proving that $r_3(n,n)$ grows at least double exponential in $\Omega(n)$, would imply that $r_k(n,n)$ does indeed grow as a tower of height $k$ in $\Omega(n)$, settling Conjecture \ref{conj1}.  In the other direction, any improvement on the upper bound for $r_3(n,n)$ can be used with a result of Erd\H os and Rado \cite{ER}, to obtain improved upper bounds for $r_k(n,n)$, for all $k\geq 4$.  It is widely believed that Conjecture \ref{conj1} is true, based on the fact that such bounds are known for four colors. More precisely, the $q$-color Ramsey number,

$$r_k(\underbrace{n,\ldots, n}_{q\textnormal{ times}})$$

\noindent is the minimum $N$ such that every $q$-coloring of the edges of the complete $N$-vertex $k$-uniform hypergraph $K^{(k)}_N$, contains a monochromatic copy of $K_n^{(k)}$.  A result of Erd\H os and Hajnal \cite{EH72} shows that $r_3(n,n,n,n) > 2^{2^{\Omega(n)}}$, and this implies that
$$r_k(n,n,n,n) = \twr_k(\Theta(n)),$$
for all fixed $k\geq 4$.  For three colors, Conlon, Fox, and Sudakov \cite{CFS} showed that for fixed $k\ge 3$,
$$r_k(n,n,n) \geq \twr_k(\Omega(\log^2 n)).$$

\subsection{Off-diagonal Ramsey numbers.}

The \emph{off-diagonal} Ramsey number, $r_k(s,n)$ with $k,s$ fixed and $n$ tending to infinity, has also been extensively studied.  It is known \cite{AKS,Kim,B,BK} that $r_2(3,n) =\Theta(n^2/\log n)$ and, for fixed $s > 3$,
\begin{equation*}\label{offgraph}
c_1\left( \frac{n}{\log n} \right)^{\frac{s+1}{2}} \leq r_2(s,n) \leq c_2 \frac{n^{s-1}}{\log^{s-2}n} ,\end{equation*}

\noindent  where $c_1$ and $c_2$ are absolute constants.  For 3-uniform hypergraphs, a result of Conlon, Fox, and Sudakov \cite{CFS} shows that for fixed $s\geq 4$  $$2^{\Omega(n\log n)} \leq r_3(s,n) \leq 2^{n^{s-2}\log n}.$$
For fixed  $s>k\geq 4$, it is known that
$$r_k(s,n) \leq \twr_{k-1}(O(n^{s-2}\log n)).$$
  By applying the Erd\H os-Hajnal stepping up lemma in the off-diagonal setting, it follows that $r_k(s,n) \geq \twr_{k-1}(\Omega(n))$, for $k\geq 4$ and for all $s \geq 2^{k-1} - k + 3$.  In 1972, Erd\H os and Hajnal  conjectured the following.

\begin{conjecture}{\bf (Erd\H os-Hajnal~\cite{EH72})}\label{conj2}
For $s \geq k + 1 \ge 5$ fixed,  $r_k(s,n) \geq\twr_{k-1}(\Omega(n)).$

\end{conjecture}

\noindent In \cite{CFS13}, Conlon, Fox, and Sudakov modified the Erd\H os-Hajnal stepping-up lemma to show that Conjecture \ref{conj2} holds for all $s\geq \lceil 5k/2\rceil - 3$.    Using a result of Duffus et al.~\cite{DLR} (see also  Moshkovitz and Shapira \cite{MS}  and Milans et al. \cite{MSW}), one can show that $r_{k}(s,n) \geq \twr_{k-2}(\Omega(n))$ for all $s\geq k + 1$.
In this paper, we prove the following result that nearly settles Conjecture~\ref{conj2} by determining the correct tower growth rate for $s\ge k+3$, and obtaining new bounds for the two remaining cases.

\begin{theorem} \label{main0}
There is a positive constant $c>0$ such that the following holds.  For $k \geq 4$ and $n>3k$,  we have
\begin{enumerate}
\item $r_k(k + 3,n) \geq \twr_{k-1}(cn),$

\item $r_k(k + 2,n) \geq \twr_{k - 1}(c\log^2 n),$

\item $r_k(k + 1, n) \geq \twr_{k - 2}(cn^2).$

\end{enumerate}

\end{theorem}

There are two novel ingredients to our constructions. First, we relate these problems to   estimates for Ramsey numbers of tight-paths versus cliques, which we find of independent interest.  Second, we use $(k-1)$-uniform diagonal Ramsey numbers for more than two colors to obtain constructions for $k$-uniform off-diagonal Ramsey numbers for two colors. This differs from the usual paradigm in this area, exemplified by the stepping up lemma, where the number of colors stays the same or goes up as the uniformity increases (see, e.g. \cite{AGLM, CFS13, CFS, DLR,  EH72, EHR, K, M15}). This topic has also been extensively studied in the context of partition relations for ordinals.  It is quite possible that our constructions can also be applied to the infinite setting, though we have not explored this here.

After this paper was written, we learned that a bound
similar to Theorem~\ref{main0} part (1) was also recently proved by Conlon, Fox, and Sudakov (unpublished), using the more traditional stepping-up argument of Erd\H os and Hajnal.

\subsection{Tight-path versus clique.}
 Consider an \emph{ordered} $N$-vertex $k$-uniform hypergraph $H$, that is, a hypergraph whose vertex set is $[N] = \{1,2,\ldots, N\}$.  A \emph{tight path} of size $s$ in $H$, denoted by $P^{(k)}_{s}$, comprises a set of $s$ vertices $v_1, \ldots, v_{s} \in [N]$, $v_1 < \cdots < v_s$, such that $(v_j,v_{j+1},\ldots, v_{j + k -1}) \in E(H)$ for $j = 1,2,\ldots, s - k + 1$.  The \emph{length} of $P^{(k)}_s$ is the number of edges, $s - k + 1$.

Here, we obtain lower and upper bounds for Ramsey numbers for tight-paths versus cliques.  To be more precise, we need the following definition.  Given $q$ $k$-uniform hypergraphs $F_1,\ldots, F_q$, the Ramsey number $r(F_1,\ldots, F_q)$ is the minimum $N$ such that every $q$-coloring of the edges of the complete $N$-vertex $k$-uniform hypergraph $K^{(k)}_N$, whose vertex set is $[N] = \{1,\ldots, N\}$, contains an $i$-colored copy of $F_i$ for some $i$.  In order to avoid the excessive use of superscripts, we use the simpler notation
$$r_k(P_s, P_n) = r(P_s^{(k)}, P_n^{(k)}) \qquad{ {\rm and} } \qquad r_k(P_s, n) = r(P_s^{(k)}, K_n^{(k)}).$$
Two famous theorems of Erd\H os and Szekeres in \cite{ES35}, known as the monotone subsequence theorem and the cups-caps theorem, imply that $r_2(P_s , P_n) = (n - 1)(s - 1) + 1$ and $r_3(P_s,P_n) = {n+s - 4 \choose s-2} + 1.$   Fox, Pach, Sudakov, and Suk \cite{FPSS} later extended their results to $k$-uniform hypergraphs, and gave a geometric application related to the Happy Ending Theorem.\footnote{The main result in \cite{ES35}, known as the Happy Ending Theorem, states that for any positive integer $n$, any sufficiently large set of points in the plane in general position has a subset of $n$ members that form the vertices of a convex polygon.}  See also \cite{DLR,MS,MSW} for related results.

The proof of the Erd\H os-Szekeres monotone subsequence theorem \cite{ES35} (see also Dilworth's Theorem \cite{D50}) actually implies that $r_2(P_s , n) = (n - 1)(s - 1) + 1.$  For $k\geq 3$, estimating $r_k(P_s, n)$ appears to be more difficult.  Clearly we have

\begin{equation} \label{easy} r_k(P_s, n) \leq r_k(s,n)   \leq \twr_{k - 1}(O(n^{s - 2}\log n)).\end{equation}

Our main result is a new connection between
$r_k(P_s, n)$ and $r_k(n,n)$. Again, we will use the simpler notation
$$r_k(n;q) = r_k(\underbrace{n,\ldots, n}_{q\textnormal{ times}}) \qquad {\rm and } \qquad
r_k(P_{s_1}, \ldots, P_{s_t}, n)=r(P^{(k)}_{s_1},\ldots, P^{(k)}_{s_t},K^{(k)}_n).$$

\begin{theorem} {\bf (Main Result)}\label{main}
Let $k\geq 2$ and $s_1,\ldots, s_t \geq k + 1$.  Then for $q = (s_1 - k + 1)\cdots (s_t - k + 1)$, we have $$r_{k -1}(\lfloor n/q\rfloor ;q)
\le r_k(P_{s_1}, \ldots, P_{s_t}, n) \le
r_{k -1}(n;q)
.$$
\end{theorem}

Theorem \ref{main} has several immediate consequences with $t=1$. First, we can considerably improve the upper bound for $r_k(P_s, n)$ in (\ref{easy}).

\begin{corollary}\label{cor1}
For fixed $k\geq 3$ and $s\geq k + 1$, we have $ r_k(P_s, n) \leq \twr_{k-1}(O(sn\log s))$.
\end{corollary}

\noindent
Indeed, using the standard Erd\H os-Szekeres recurrence~\cite{ES35}, we have $r_2(n;q)< q^{nq}=\twr_2(O(qn\log q))$, and the upper bound argument of Erd\H os-Rado~\cite{ER} then yields
$r_{k-1}(n;q)< \twr_{k-1}(O(qn\log q))$. Applying Theorem~\ref{main} with $t=1, s_1=s$, and $q=s-k+1<s$, now implies Corollary~\ref{cor1}.
\medskip

Combining the lower bounds in Theorem \ref{main} with the known lower bounds for $r_{k-1}(n,n, n,n)$ in~\cite{EH72}, $r_{k - 1}(n,n,n)$ in~\cite{CFS}, and $r_{k - 1}(n,n)$ in~\cite{EH72}, we establish the following inequalities.  There is an absolute constant $c>0$ such that for all $k \geq 4$ and $n>3k$
$$ r_k(P_{k  + 3}, n)  \geq r_{k-1}\left(\frac{n}{4},\frac{n}{4},\frac{n}{4},\frac{n}{4}\right) \geq \twr_{k-1}(cn),$$

$$ r_k(P_{k  + 2}, n)  \geq r_{k-1}\left(\frac{n}{3},\frac{n}{3},\frac{n}{3}\right) \geq \twr_{k-1}(c\log^2n),$$

$$r_k(P_{k  + 1}, n)  \geq r_{k-1}\left(\frac{n}{2},\frac{n}{2}\right)  \geq \twr_{k-2}(cn^2).$$
Summarizing, we have just proved parts 1--3 of the following theorem, which is a strengthening of Theorem~\ref{main0} as $r_k(s,n) \ge r_k(P_s, n)$.

\begin{theorem}\label{mainlower}
There is a positive constant $c>0$ such that $r_3(P_4,n) > 2^{cn}$, and for $k \geq 4$ and $n>3k$,
\begin{enumerate}

\item $r_k(P_{k + 3}, n) \geq \twr_{k-1}(cn),$

\item $r_k(P_{k + 2}, n) \geq \twr_{k - 1}(c\log^2 n),$

\item  $r_k(P_{k + 1}, n) \geq \twr_{k - 2}(cn^2).$

\end{enumerate}
\end{theorem}

We conjecture the following strengthening of the Erd\H os-Hajnal conjecture.

\begin{conjecture}\label{conj3}
For $k\geq 4$ fixed, $r_k(P_{k + 1}, n)  \ge \twr_{k-1}(\Omega(n)).$
\end{conjecture}

For $t = 1$, $q = 2$, and $s_1 = k + 1$ in Theorem \ref{main}, we have $r_{k-1}(n/2,n/2) \leq r_k(P_{k + 1},n) \leq r_{k-1}(n,n)$.  Hence, we obtain the following corollary which relates $r_k(P_{k+1}, n)$ to the diagonal Ramsey number $r_k(n,n)$.

\begin{corollary}\label{equiv}
Conjecture \ref{conj1} holds if and only if Conjecture \ref{conj3} holds.
\end{corollary}
Corollary~\ref{equiv} shows that our lack of understanding of the Ramsey number $r_k(P_{k + 1}, n)$ is due to our lack of understanding of the diagonal Ramsey number $r_{k -1}(n,n)$.  However, if we add one additional color, then Theorem~\ref{main} with $t=2$ implies that $r_k(P_{k + 1}, P_{k + 1}, n)$ does indeed grow as a tower of height $k-1$ in $\Omega(n)$.

\begin{corollary} \label{3colors} There is a positive constant $c>0$ such that for $k \geq 4$ and $n>3k$,
$$r_k(k + 1,k + 1,n) \geq r_k(P_{k  + 1}, P_{k  + 1}, n) \geq r_{k-1}\left(\frac{n}{4},\frac{n}{4},\frac{n}{4},\frac{n}{4}\right) \geq \twr_{k-1}(cn).$$
\end{corollary}

Note that by the results of Erd\H os and Rado~\cite{ER}, for every $k\ge 4$, there is an $c_k>0$
such that $r_k(k + 1,k + 1,n)\le \twr_{k-1}(n^{c_k})$.

In the next Section, we prove Theorem~\ref{main} and the inequality $r_3(P_4,n) > 2^{\Omega(n)}$ from Theorem~\ref{mainlower}.  In Sections \ref{secondlast} and \ref{last}, we study several related Ramsey problems.  We sometimes omit floor and ceiling signs whenever they are not crucial for the sake of clarity of presentation.

\section{Ramsey numbers for tight paths versus cliques}\label{pathvsclique}

In this section, we prove Theorem~\ref{main}.

{\bf Proof of Theorem~\ref{main}.}
Let us first prove the upper bound.  Set $q_i = s_i - k + 1$ so that $q = q_1\cdots q_t$, and $N = r_{k -1}(n;q)$.  Let $\chi:{[N]\choose k} \rightarrow \{1,2,\ldots, t + 1\}$ be a $(t + 1)$-coloring of the edges of $K^{(k)}_N$.  We will show that $\chi$ must produce a monochromatic copy of $P^{(k)}_{s_i}$ in color $i$, for some $i$, or a monochromatic copy of $K^{(k)}_n$ in color $t + 1$.

Define $\phi:{[N]\choose k-1} \rightarrow \mathbb{Z}^t$, where for $v_1,\ldots, v_{k - 1} \in [N]$, $v_1 <  \cdots < v_{k-1}$, we have $\phi(v_1,\ldots, v_{k-1}) = (a_1,\ldots, a_t)$, where $a_i$ is the \emph{length} of the longest monochromatic tight-path in color $i$ ending with $v_1,\ldots, v_{k-1}$.  If $a_i \geq q_i$ for some $i$, then we would be done since this implies we have a monochromatic copy of $P^{(k)}_{s_i}$ in color $i$.  Therefore, we can assume $a_i \in \{0,1,\ldots, q_i-1\}$ for all $i$, and hence $\phi$ uses at most $q$ colors.

Since $N = r_{k-1}(n;q)$, there is a subset $S\subset [N]$ of $n$ vertices such that $\phi$ colors every $(k-1)$-tuple in $S$ the same color, say with color $(b_1,\ldots, b_t)$.  Then notice that for every $k$-tuple $(v_1,\ldots, v_k) \in {S\choose k}$, we have $\chi(v_1,\ldots, v_k) = t + 1$.  Indeed, suppose there are $k$ vertices $v_1,\ldots, v_k \in S$ such that $\chi(v_1,\ldots, v_k) = i$, where $i \leq t$.  Since the longest monochromatic tight-path in color $i$ ending with vertices $v_1,\ldots, v_{k-1}$ is $b_i$, this implies that the longest monochromatic tight-path in color $i$ ending with vertices $v_2,\ldots, v_k$ is at least $b_i + 1$, a contradiction.  Therefore, $S$ induces a monochromatic copy of $K^{(k)}_n$ in color $t + 1$.   This concludes the proof of the upper bound
\medskip

We now prove the lower bound. Set $N = r_{k -1}(\lfloor n/q\rfloor ;q) - 1$ and $q_i = s_i - k + 1$, so that $q = q_1\cdots q_t$. Let $K^{(k-1)}_N$ be the complete $N$-vertex $(k-1)$-uniform hypergraph with vertex set $[N]$. Next, let $$\phi:{{N}\choose k-1}\rightarrow [q_1]\times\cdots \times [q_t]$$ be a $q$-coloring on the edges of $K^{(k-1)}_N$, that does not produce a monochromatic copy of $K^{(k-1)}_{\lfloor n/q\rfloor}$.  Such a coloring $\phi$ exists since $N = r_{k -1}(\lfloor n/q\rfloor ;q) - 1$. We now define a $(t + 1)$-coloring
$$\chi:{[N]\choose k} \rightarrow [t+1]$$ on the $k$-tuples of $[N]$ as follows. For $v_1,\ldots, v_k \in [N]$, where $v_1 < \cdots < v_k$, let $\chi(v_1,\ldots, v_k) = i$ if and only if for $\phi(v_1,\ldots, v_{k - 1}) = (a_1,\ldots, a_{t})$ and $\phi(v_2,\ldots, v_k) = (b_1,\ldots, b_{t})$, $i$ is the minimum index such that $a_i < b_i$.  If no such $i$ exists, then $\chi(v_1,\ldots, v_k) = t+1$.  We will show that $\chi$ does not produce a monochromatic $i$-colored copy of $P^{(k)}_{s_i}$, for $i \leq t$, nor a monochromatic $(t+1)$-colored copy of $K^{(k)}_{n}$.

First, suppose that the coloring $\chi$ produces a monochromatic $P_{s_i}^{(k)}$ in color $i$.  That is, there are $s_i$ vertices $v_1, v_2, \ldots, v_{s_i} \in [N]$, $v_1 < \cdots < v_{s_i}$, such that $\chi(v_j,v_{j + 1},\ldots, v_{j + k - 1}) = i$ for $j = 1,\ldots, s_i -k  + 1$.  Let $\phi(v_j,v_{j + 1},\ldots, v_{j + k - 2}) = (a_{j,1},\ldots, a_{j,t})$, for $j = 1,\ldots, s_i - k + 2$.  Then we have

$$a_{1,i} < a_{2,i} < \cdots < a_{s_i - k + 2, i},$$
 which is a contradiction since $q_i  < s_i - k + 2$.  Hence, $\chi$ does not produce a monochromatic $P^{(k)}_{s_i}$ in color $i$ for $i \leq t$.

Next, we show that $\chi$ does not produce a monochromatic copy of $K^{(k)}_n$ in color $t + 1$.  Again, for sake of contradiction, suppose there is a set $S\subset [N]$ where $S = \{v_1 , \ldots, v_n\}$, $v_1 < \cdots < v_n$, such that $\chi$ colors every $k$-tuple of $S$ with color $t + 1$.  We obtain a contradiction from the following claim.

\begin{claim}

Let $S = \{v_1,\ldots, v_n\}$, $\chi$, and $\phi$ be as above, and $1 \leq \ell \leq q$.  If $\phi$ uses at most $\ell$ distinct colors on ${S\choose k-1}$, and if $\chi$ colors every $k$-tuple of $S$ with color $t  + 1$, then there is a subset $T\subset S$ of size $\lfloor n/\ell\rfloor$ and a color $a=(a_1,\ldots, a_t)$
such that $\phi(T')=a$ for every
$T' \in {T\choose k-1}$.

\end{claim}

\noindent The contradiction follows from the fact that $\lfloor n/\ell \rfloor \geq \lfloor n/q\rfloor$, and $\phi$ does not produce a monochromatic copy of $K^{(k-1)}_{\lfloor n/q\rfloor}$.

\medskip

 \emph{Proof of Claim}.  We proceed by induction on $\ell$.  The base case $\ell = 1$ is trivial.  For the inductive step, assume that the statement holds for $\ell' < \ell$.  Let $\mathcal{C}$ be the set of $\ell$ distinct colors defined by $\phi$ on ${S\choose k-1}$, and let $(a_1^{\ast},\ldots, a^{\ast}_t) \in \mathcal{C}$ be the smallest element in $\mathcal{C}$ with respect to the lexicographic ordering.  We set $S_1 = \{v_1,\ldots, v_{ n - \lfloor n/\ell\rfloor}\}$ and $S_2 = \{v_{ n - \lfloor n/\ell\rfloor + 1}  , \ldots, v_n\}$. The proof now falls into two cases.

\medskip

\noindent \emph{Case 1.}  Suppose there is a $(k-1)$-tuple $(u_1,\ldots, u_{k-1}) \in {S_1\choose k-1}$ such that $\phi(u_1,\ldots, u_{k-1}) = (a^{\ast}_1,\ldots, a^{\ast}_t)$.  Then we have $\phi(T')= (a^{\ast}_1,\ldots, a^{\ast}_t)$ for all
$T' \in {S_2\choose k-1}$.  Indeed let $T'=(w_1,\ldots, w_{k-1}) \in {S_2\choose k-1}$.  Since $\chi(u_1,\ldots, u_{k-1},w_1) = t + 1$, we have $\phi(u_2,\ldots, u_{k-1}, w_1) = (a^{\ast}_1,\ldots, a^{\ast}_t)$.  Likewise, since we have $\chi(u_2,\ldots, u_{k-1},w_1,w_2) = t + 1$, we have $\phi(u_3,\ldots, u_{k-1}, w_1, w_2) = (a^{\ast}_1,\ldots, a^{\ast}_t)$.  By repeating this argument $k - 3$ more times, $\phi(w_1,\ldots, w_{k-1}) = (a^{\ast}_1,\ldots, a^{\ast}_t)$.

\medskip

\noindent \emph{Case 2.}  If we are not in Case 1, then $\phi(T') \in \mathcal{C}\setminus \{(a^{\ast}_1,\ldots, a^{\ast}_t)\}$ for every $T' \in {S_1 \choose k-1}$.  Hence $\phi$ uses at most $\ell - 1$ distinct colors on ${S_1\choose k-1}$.  By the induction hypothesis, there is a subset $T \subset S_1$ of size $(n - \lfloor n/\ell\rfloor)/(\ell - 1) \geq \lfloor n/\ell\rfloor$
and a color $a=(a_1,\ldots, a_t)$ such that $\phi(T')=a$ for every $T' \in {T \choose k-1}$.
This concludes the proof of the claim and the theorem. \qed

\bigskip

\noindent \textbf{Lower bound construction for $r_3(P_4, n)$ in Theorem~\ref{mainlower}}.  Set $N = 2^{cn}$ where $c$ will be determined later.  Consider the coloring $\phi:{[N]\choose 2} \rightarrow \{1,2\}$, where each edge has probability $1/2$ of being a particular color independent of all other edges.  Using $\phi$, we define the coloring $\chi:{[N]\choose 3}\rightarrow \{\textnormal{red},\textnormal{blue}\}$, where the triple $(v_1,v_2,v_3)$, $v_1 < v_2 < v_3$, is red if $\phi(v_1,v_2) < \phi(v_2,v_3)$, and is blue otherwise.  It is easy to see that $\chi$ does not produce a monochromatic red copy of $P^{(3)}_4$.

Next we estimate the expected number of monochromatic blue copies of $K^{(k)}_n$ in $\chi$.  For a given triple $\{v_1,v_2,v_3\} \in {[N]\choose 3}$, the probability that $\chi(v_1,v_2,v_3) = \textnormal{blue}$ is  $3/4$.  Let $T = \{v_1,\ldots, v_n\}$ be a set of $t$ vertices in $[N]$, where $v_1 < \cdots < v_n$.  Let $S$ be a partial Steiner $(n,3,2)$-system with vertex set $T$, that is, $S$ is a 3-uniform hypergraph such that each $2$-element set of vertices is contained in at most one edge in $S$.  Moreover, $S$ satisfies $|S| = c'n^{2}$.  It is known that such a system exists. Then the probability that every triple in $T$ is blue is at most the probability that every triple in $S$ is blue.  Since the edges in $S$ are independent, that is no two edges have more than one vertex in common, the probability that $T$ is a monochromatic blue clique is at most $\left( \frac{3}{4}\right)^{|S|} \leq \left(\frac{3}{4}\right)^{c'n^{2}}$.  Therefore the expected number of monochromatic blue copies of $K^{(k)}_n$ produced by $\chi$ is at most

$${N\choose n}\left( \frac{3}{4}\right)^{c'n^{2}} < 1,$$

\noindent for an appropriate choice for $c$.  Hence, there is a coloring $\chi$ with no monochromatic red copy of $P_4^{(3)}$, and no monochromatic blue copy of $K^{(k)}_n$.  Therefore

$$r_3(P_4,n) > 2^{\Omega(n)}.$$

\section{A Ramsey-type result for nonincreasing sets}\label{secondlast}

 Notice that the proof of Theorem \ref{main} does not require the full strength of Ramsey's theorem.  For example when $t = 1$, rather than finding $n$ vertices with the property that every $(k-1)$-tuple has the same color, it is enough to find a set $T$ of $n$ vertices such that for any subset of $k$ vertices $v_1, \ldots, v_{k} \in T$, where $v_1 < \cdots < v_{k}$, we have $\phi(v_1,\ldots, v_{k-1}) \geq \phi(v_2,\ldots, v_{k })$.

Motivated by the observation above, we study the following variant of $r_k(n;q)$.  Let $\chi:{[N]\choose k}\rightarrow\{1,\ldots, q\}$ be a $q$-coloring on the $k$-tuples of $[N]$ with colors $\{1,2,\ldots, q\}$.  For $T\subset [N]$, we say that $T$ is \emph{nonincreasing}, if for any set of $k + 1$ vertices $v_1, \ldots, v_{k + 1} \in T$, where $v_1 < \cdots < v_{k + 1}$, we have $\chi(v_1,\ldots, v_k) \geq \chi(v_2,\ldots, v_{k  + 1})$.  Let $f_k(n;q)$ be the minimum integer $N$, such that for any $q$-coloring on the $k$-tuples of $[N]$, with colors $\{1,2,\ldots, q\}$, contains a nonincreasing set $T$ of size $n$.  Clearly we have $f_k(n;q) \leq r_k(n;q)$, and recall that the best known upper bound for $r_2(n,n)$ is $4^{n-o(n)}$.  Our next result makes the following improvement in this fundamental case.

\begin{theorem}
We have $f_2(n;2) \leq \left\lceil\left(2 + \sqrt{2}\right)^n\right\rceil \approx (3.414)^n$.

\end{theorem}

\begin{proof}

We proceed by induction on $n$.  The base case $n  =1$ is trivial.  Suppose now the statement holds for $n' < n$.  Set $N = \left\lceil\left(2 + \sqrt{2}\right)^n\right\rceil$, and let $\chi:{[N]\choose 2} \rightarrow \{1,2\}$.  We will show that there is a nonincreasing subset of size $n$.  Suppose there is a vertex $v \in [N]$ and a subset $S_v\subset \{1,\ldots, v-1\}$ such that $|S_v| \geq \left\lceil\left(2 + \sqrt{2}\right)^{n-1}\right\rceil$, and for every $u \in S_v$ we have $\chi(u,v) = 1$.  By the induction hypothesis, $S_v$ contains a nonincreasing set of size $n-1$, and together with $v$ we have a nonincreasing set of size $n$ and we are done.  Therefore we can assume no such vertex $v \in [N]$ exist.

Set $L = \left\lceil\left(2 + \sqrt{2}\right)^{n-1}\right\rceil$, and let $E_2\subset {[N]\choose 2}$ denote the set of pairs in $[N]$ with color 2, and whose left endpoint lies in $\{1,\ldots, N - L\}$.  For each $v \in [N]$, let $d_2(v)$ denote the number of edges in $E_2$ whose right endpoint is $v$.  By the assumption above, the back degree of each $v \in [N]$ in color 1 is at most $L - 1$, which implies $d_2(v) \geq v - 1 - (L - 1) = v - L$. Thus we have

$$\begin{array}{ccl}
    |E_2| & \geq &  \sum\limits_{v = L + 1}^N d_2(v)   \\\\
      & \geq & 1 + 2 + \cdots + (N - 2L + 2) + (L - 1) (N - 2L + 2)   \\\\
      & \geq & \frac{N(N - 2L + 2)}{2}.
  \end{array}$$

\noindent By the pigeonhole principle, there is a vertex $u \in \{1,\ldots,  N - L\}$ and a set $T\subset \{u + 1,\ldots, N\}$, such that

$$\begin{array}{ccl}
    |T| & \geq &  \frac{N(N - 2L + 2)}{2(N -L)} \\\\
     &  =  & \frac{N(N - 2L)}{2(N - L)} + \frac{N}{N - L} \\\\
     & \geq  & \left(\frac{\sqrt{2}}{1 + \sqrt{2}}\frac{N}{2}\right)  + \left(\frac{2 + \sqrt{2}}{1 + \sqrt{2}}\right) \\\\
 & \geq & (2 + \sqrt{2})^{n - 1} + \left(\frac{2 + \sqrt{2}}{1 + \sqrt{2}}\right),
  \end{array}$$

\noindent and for all $v \in T$ we have $\chi(u,v) = 2$.  Hence, $|T|\geq \left\lceil\left(2 + \sqrt{2}\right)^{n-1}\right\rceil$.  By the induction hypothesis, we can find a nonincreasing set in $T$ of size $n-1$, and together with $u$ we have a nonincreasing set of size $n$. \end{proof}

\begin{corollary}
We have $r_3(P_4, n) \leq \left\lceil\left(\frac{2}{2 - \sqrt{2}}\right)^n\right\rceil \approx (3.414)^n$.

\end{corollary}

\medskip

\section{A related off-diagonal problem}\label{last}

Let $K_4^{(3)}\setminus e$ denote the 3-uniform hypergraph on four vertices, obtained by removing one edge from $K_4^{(3)}$.  A simple argument of Erd\H os and Hajnal \cite{EH72}  implies $r(K_4^{(3)}\setminus e,K_n^{(3)}) < (n!)^2$.   Here, we generalize their argument to establish an upper bound for Ramsey numbers for $k$-half-graphs versus cliques.  A \emph{$k$-half-graph}, denote by $B^{(k)}$, is a $k$-uniform hypergraph on $2k-2$ vertices, whose vertex set is of the form $S\cup T$, where $|S| = |T| = k-1$, and whose edges are all $k$-subsets that contain $S$, and one $k$-subset that contains $T$.  So $B^{(3)}=K_4^{(3)}\setminus e$.
The goal of this section is to obtain upper and lower bounds for $r(B^{(k)},K^{(k)}_n)$.  We start with the upper bound.

\begin{theorem}\label{halfup}
For $k\geq 4$, we have $ r(B^{(k)},K^{(k)}_n) \leq (n!)^{k-1}.$

\end{theorem}

\noindent First, let us recall an old lemma due to Spencer.

\begin{lemma}[\cite{S}]\label{spencer}
Let $H = (V,E)$ be a $k$-uniform hypergraph on $N$ vertices.  If $|E(H)| > N/k$, then there exists a subset $S\subset V(H)$ such that $S$ is an independent set and

$$|S| \geq \left(1 - \frac{1}{k}\right)N \left(\frac{N}{k|E(H)|}\right)^{\frac{1}{k-1}}.$$

\end{lemma}

\medskip

\noindent \emph{Proof of Theorem \ref{halfup}.}  We proceed by induction on $n$.  The base case $n = k$ is trivial.  Let $n > k$ and assume the statement holds for $n'  < n$.  Let $k\geq 4$ and let $\chi$ be a red/blue coloring on the edges of $K^{(k)}_N$, where $N = (n!)^{k-1}$.  Let $E_R$ denote the set of red edges in $K^{(k)}_N$.

\medskip

\noindent \emph{Case 1}:  Suppose $|E_R| < \frac{\left(1 - \frac{1}{k}\right)^{k-1}N^k}{ k n^{k-1}  }.$ Then by Lemma \ref{spencer}, $K^{(k)}_N$ contains a blue clique of size $n$.

\medskip

\noindent \emph{Case 2}:  Suppose $|E_R| \geq \frac{\left(1 - \frac{1}{k}\right)^{k-1}N^k}{ k n^{k-1}  }.$  Then by averaging, there is a $(k-1)$-element subset $S\subset V$ such that $N(S) = \{v \in V: S\cup \{v\} \in E_R\}$ satisfies

$$|N(S)| \geq \frac{\left(1 - \frac{1}{k}\right)^{k-1}N^k}{ n^{k-1}{n\choose k-1} }  \geq \left((n-1)!\right)^{k-1} .$$

\noindent  The last inequality follows from the fact that $k\geq 4$.  Fix a vertex $u \in S$.  If $\{u\}\cup T \in E_R$ for some $T\subset N(S)$ such that $|T| = k-1$, then $S\cup T$ forms a red $B^{(k)}$ and we are done.  Therefore we can assume otherwise.  By the induction hypothesis, $N(S)$ contains a red copy of $B^{(k)}$, or a blue copy of $K^{(k)}_{n-1}$.  We are done in the former case, and the latter case, we can form a blue $K_{n}^{(k)}$ by adding the vertex $u$. $\hfill\square$

\medskip

We now give constructions which show that $r(B^{(k)},K^{(k)}_n)$ is at least exponential in $n$.

\begin{theorem}\label{halflow}
For fixed $k\geq 3$, we have $ r(B^{(k)},K^{(k)}_n) > 2^{\Omega(n)}.$

\end{theorem}

\begin{proof} Surprisingly, we require different arguments for $k$ even and $k$ odd.

\medskip

\noindent \emph{The case when $k$ is odd.}  Assume $k$ is odd, and set $N = 2^{cn}$ where $c = c(k)$ will be determined later.  Then let $T$ be a random tournament on the vertex set $ [N]$, that is, for $i,j \in [N]$, independently, either $(i,j) \in E$ or $(j,i) \in E$, where each of the two choices is equally likely.  Then let $\chi:{[N]\choose k} \rightarrow \{\textnormal{red},\textnormal{blue}\}$ be a red/blue coloring on the $k$-subsets of $[N]$, where $\chi(v_1,\ldots, v_k) =$ red if $v_1,\ldots, v_k$ induces a \emph{regular} tournament, that is, the indegree of every vertex is $(k-1)/2$ (and hence the outdegree of every vertex is $(k-1)/2$).  Otherwise we color it blue.  We note that since $k$ is odd, a regular tournament on $k$ vertices is possible by the fact that $K_{k}$ has an Eulerian circuit, and then by directing the edges according to the circuit we obtain a regular tournament.

Notice that the coloring $\chi$ does not contain a red $B^{(k)}$.  Indeed, let $S,T \subset [N]$ such that $|S| = |T| = k-1$, $S\cap T = \emptyset$, and every $k$-tuple of the form $S\cup \{v\}$ is red, for all $v \in T$.  Then for any $u \in S$, all edges in the set $u\times T$ must have the same direction, either all emanating out of $u$ or all directed towards $u$.  Therefore it is impossible for $u\cup T$ to have color red, for any choice $u \in S$.

Next we estimate the expected number of monochromatic blue copies of $K^{(k)}_n$ in $\chi$.  For a given $k$-tuple $v_1,\ldots,v_k \in [N]$, the probability that $\chi(v_1,\ldots,v_k) = \textnormal{blue}$ is clearly at most $1 - 1/2^{{k\choose 2}}$.  Let $T = \{v_1,\ldots, v_n\}$ be a set of $t$ vertices in $[n]$, where $v_1 < \cdots < v_n$.  Let $S$ be a partial Steiner $(n,k,2)$-system with vertex set $T$, that is, $S$ is a $k$-uniform hypergraph such that each $2$-element set of vertices is contained in at most one edge in $S$.  Moreover, $S$ satisfies $|S| = c_kn^{2}$.  It is known that such a system exists. Then the probability that every $k$-tuple in $T$ has color blue is at most the probability that every $k$-tuple in $S$ is blue.  Since the edges in $S$ are independent, that is no two edges have more than one vertex in common, the probability that $T$ is a monochromatic blue clique is at most $\left(1 - 1/2^{{k\choose 2}}\right)^{|S|} \leq \left(1 - 1/2^{{k\choose 2}}\right)^{c_kn^{2}}$.  Therefore the expected number of monochromatic blue copies of $K^{(k)}_n$ in $\chi$ is at most

$${N\choose n}\left(1 - 1/2^{{k\choose 2}}\right)^{c_kn^{2}} < 1,$$

\noindent for an appropriate choice for $c = c(k)$.  Hence, there is a coloring $\chi$ with no red $B^{(k)}$ and no blue $K^{(k)}_n$.  Therefore

$$r(B^{(k)},K_n^{(k)}) > 2^{cn }.$$

\medskip

\noindent \emph{The case when $k$ is even.}   Assume $k$ is even and set $N = 2^{ct}$ where $c = c(k)$ will be determined later.  Consider the coloring $\phi:{[N]\choose 2} \rightarrow \{1,\ldots, k-1\}$, where each edge has probability $1/(k-1)$ of being a particular color independent of all other edges (pairs).  Using $\phi$, we define the coloring $\chi:{[N]\choose k}\rightarrow \{\textnormal{red},\textnormal{blue}\}$, where the $k$-tuple $(v_1,\ldots, v_k)$ is red if $\phi$ is a proper edge-coloring on all pairs among $\{v_1,\ldots, v_k\}$, that is, each of the $k-1$ colors appears as a perfect matching.  Otherwise we color it blue.

Notice that the coloring $\chi$ does not contain a red $B^{(k)}$.  Indeed let $S,T \subset [N]$ such that $|S| = |T| = k-1$ and $S\cap T = \emptyset$.  If every $k$-tuple of the form $S\cup \{v\}$ is red, for all $v \in T$, then all $k-1$ colors from $\phi$ appear among the edges (pairs) in the set $S\times \{v\}$.  Hence for any vertex $u\in S$, $\chi$ could not have colored $u\cup T$ red since it is impossible to have \emph{any} of the $k-1$ colors to appear as a perfect matching in $u\cup T$.

For a given $k$-tuple $v_1,\ldots,v_k \in [N]$, the probability that $\chi(v_1,\ldots,v_k) = \textnormal{blue}$ is at most $1 - (1/(k-1))^{{k\choose 2}}$.  By the same argument as above, the expected number of monochromatic blue copies of $K^{(k)}_n$ with respect to $\chi$ is less than 1 for an appropriate choice of $c = c(k)$.  Hence, there is a coloring $\chi$ with no red $B^{(k)}$ and no blue $K^{(k)}_n$.  Therefore

$$r(B^{(k)},K_n^{(k)}) > 2^{cn}$$
and the proof is complete.
\end{proof}

\section{Acknowledgments}
We thank David Conlon and Jacob Fox for comments that helped improve the presentation of this paper.

\end{document}